\documentclass{article}
\pdfoutput=1
\usepackage[utf8]{inputenc}

\usepackage{graphicx}

\usepackage{amsthm}
\usepackage{amsmath}
\usepackage{amsfonts}
\usepackage{amssymb}
\usepackage[english]{babel}

\usepackage{hyperref}

\theoremstyle{plain}
\newtheorem{theorem}{Theorem}
\newtheorem{proposition}[theorem]{Proposition}

\newtheorem{corollary}[theorem]{Corollary}

\newtheorem{remark}[theorem]{Remark}

\title{Minimal energy point systems on the unit circle and the real line}

\author{Marcell Gaál, Béla Nagy, Zsuzsanna Nagy-Csiha and Szilárd Révész}

\date{}

\newcommand*{\bD}{\mathbb{D}}
\newcommand*{\bC}{\mathbb{C}}
\newcommand*{\bR}{\mathbb{R}}

\DeclareMathOperator{\supp}{supp}

\newcommand*{\cpw}{K}

\begin{document}

\maketitle

\begin{abstract}

In this paper,
we investigate discrete logarithmic energy problems in the unit circle.
We study the equilibrium configuration of $n$ electrons and $n-1$ pairs of external protons of charge $+1/2$.
It is shown that all the critical points of the discrete logarithmic energy are global minima, and they are the solutions of certain equations involving Blaschke products.
As a nontrivial application, we refine a recent result of Simanek, namely, we prove that any configuration of $n$ electrons in the unit circle is in stable equilibrium (that is, they are not just critical points but are of minimal energy) with respect to an external field generated by $n-1$ pairs of protons.

MSC 2010 Classification codes: 31C20, 30J10, 78A30


\end{abstract}

\section{Introduction and preliminaries}

The motivation of this work comes from
certain equilibrium questions
which, in turn, have roots
in rational orthogonal systems.
Exploring the connection between critical points of 
orthogonal polynomials
and
equilibrium points
goes back to Stieltjes.
For more on this connection,
see, e.g., 
\cite{Ismail2000a}, \cite{Ismail2000b}
and the references therein.

Rational orthogonal systems are widely used on the area of signal processing, and also on the field of system and control theory.
These systems consist of rational functions with poles located outside the closed unit disk.
A wide class of rational orthogonal systems is the so-called Malmquist-Takenaka system 
from which one can recover the usual trigonometric system, the Laguerre system and the Kautz system as well.
In earlier works, 
in analogy with 
the discrete Fourier transform, a discretized version of the Malmquist-Takenaka system was introduced.

In signal processing and system identification (e.g.~ mechanical
systems related to control theory) the rational orthogonal
bases and Malmquist–Takenaka systems
(e.g.~discrete Laguerre and Kautz systems)
are more efficient than the trigonometric system
in the determination of the transfer functions.
There are lots of results in this field, see e.g.
\cite{FeichtingerPap} and the references therein, or \cite{MQW} and \cite{ModellingBook}.

In connection with potential theory,
it was studied (e.g.~in \cite{PapSchipp2001})
whether the discretization nodes satisfy certain equilibrium conditions, namely,
whether
they arise from critical points of a logarithmic potential  energy.
Such discretizations appear naturally, see e.g.~  
\cite{Bultheelbook} by Bultheel et al or 
\cite{Golinskii} by Golinskii.
The question whether the critical points are minima was proposed by Pap and Schipp \cite{PapSchipp2001, PapSchipp2015}.
In this paper, we follow this line of research.
After this introduction and statements of results, we study on the unit circle
a quite general
logarithmic energy
which is determined by a signed measure, and prove that after
inverse Cayley transform the transformed energy
on the real line differs only in an additive constant.
Next using a recent result of Semmler and Wegert \cite{SemmlerWegert} we give an affirmative answer
to the question posed by Pap and Schipp concerning the critical points.
Finally, as an application,
we present a refinement of a result of Simanek \cite{Simanek}.

First let us start with
 some
notation and essential background material.
We use the standard notations ${\mathbb D}:=\{z\in {\mathbb C} :~ |z|<1\}$, $\partial {\mathbb D}:=\{z\in {\mathbb C} :~ |z|=1\}$, ${\mathbb D}^{*}:=\{z\in {\mathbb C} :~ |z|>1\}$, ${\mathbb T}:={\mathbb R}/2\pi{\mathbb Z}$ and $\zeta^*:=1/\overline{\zeta}$ $(\zeta\ne 0)$.
We  also  use Blaschke products, defined for
$a_1,\ldots,a_n\in \bD$
and $\chi$, $|\chi|=1$
as
\begin{equation} \label{blaschkeprod}
    B(z):=
    \chi\prod_{k=1}^n \frac{z-a_k}{1-\overline{a_k}z}.
\end{equation}
In particular, when the leading coefficient $\chi=1$, $B(z)$ is called \emph{monic} Blaschke product.

We assume $B'(0)\ne 0$. In this case the well-known Walsh' Blaschke theorem (see for instance \cite{SheilSmall}, p. 377) says that $B'(z)=0$
has $2n-2$ (not necessarily different) solutions, where
$n-1$ of them (counted with multiplicites) are in the unit disk,
and if $\zeta\in\bD\setminus \{0\}$ satisfies $B'(\zeta)=0$, then $\zeta^*:=1/\overline{\zeta}$
is also a critical point, $B'( \zeta^*)=0$, with the same multiplicity as $\zeta$. It also follows that then $B'|_{\partial \mathbb{D}} \ne 0$.

Next, we investigate the structure of solutions of the equation
\begin{equation}
B(e^{it})=e^{i \delta},
\label{Blaschkeeq}
\end{equation}
where $B(.)$ is a Blaschke product. 
It is standard to see that $\Im \log B(e^{it})$ 
can be defined continuously and it is strictly increasing
on $[0,2\pi]$ from 
\[ 
\alpha:=\Im \log B(1) =\arg B(1),
\quad \alpha \in[-\pi,\pi)
\]
to $\alpha + 2n\pi$, 
see, e.g.~\cite{SheilSmall}, pp.~373-374.
Therefore \eqref{Blaschkeeq} has $n$ different solutions 
 in $t\in[0,2\pi)$ for any $\delta\in\bR$. Hence it is logical to consider $n$-tuples of different solutions as solution vectors for \eqref{Blaschkeeq}.
 
Now, we
are to reduce different types of symmetries
among the solution vectors
step-by-step. 
For given $\delta\in\bR$, consider
\begin{equation}
 \big\{(\tau_1,\ldots,\tau_n)\in \bR^n:
 B\left(e^{i \tau_j}\right)=e^{i \delta}, \
j=1,\ldots,n 
 \big\}.
 \label{allsol}
\end{equation}
We can restrict our attention to the reduced
set
$\tau_1\le\tau_2\le\ldots\le\tau_n\le \tau_1+2\pi$
without loss of generality, for picking any $\tau_1$ we can normalize $\bmod~ 2\pi$ and then order the remaining $\tau_j$.
Actually, since the $\tau_j$ are different, 
all such solutions of \eqref{Blaschkeeq} belong to the open
set
\begin{equation}
A:=\big\{\left(\tau_1,\tau_2,\ldots,\tau_n\right)\in\bR^n:\
\tau_1<\tau_2<\ldots<\tau_n<\tau_1+2\pi
\big\}.
\label{Asetdef}
\end{equation}
It is a standard step (see  \cite{SheilSmall} loc. cit.)
that one can define the functions
$\delta \mapsto \tau_j(\delta)$
such that they
are continuously differentiable,
strictly increasing,
and
$\tau_1(\delta)<\ldots<\tau_n(\delta)<\tau_1(\delta)+2\pi$
for all $\delta\in\bR$, while 
$B(\exp(i\tau_j(\delta)))=\exp(i\delta)$ $j=1,\ldots,n$.
As $B(e^{i 0})=e^{i\alpha}$,
we have
$0\in \{\tau_1(\alpha),\tau_2(\alpha),\ldots,\tau_n(\alpha)\}$. 
By relabelling again, if necessary, we may assume that
\begin{equation}
\tau_1(\alpha)=0.
\label{tau1alpha}
\end{equation}
Hence $T(\delta):=\big(\tau_1(\delta),\ldots,\tau_n(\delta)\big)$
can be viewed as a smooth arc
lying in
$A\subset \bR^n$.
Moreover, the graph $S_\bR:=\{T(\delta):\  \delta\in\bR\}$ contains
all the solutions of \eqref{Blaschkeeq}
from $A$, that is, if
$\mathbf{t}:=(t_1,\ldots,t_n)\in A$
and $\lambda\in\bR$
are such that
$B(\exp(i t_j))=\exp(i\lambda)$,
$j=1,\ldots,n$
hold,
then there exists
$\delta\in\bR$
such that
$\mathbf{t}=T(\delta)$.
Furthermore,
$\exp\big(i \tau_j(\delta+2n\pi)\big)=\exp\big(i \tau_j(\delta)\big)$ for $j=1,2,\ldots,n$, $\delta\in\bR$.
We introduce the set
\begin{equation} 
S_{0}  := S_{\bR}
\cap [0,2\pi)^n 
=\{T(\delta):\ \delta\in[\alpha,\alpha+2\pi)\}
\label{solcurve}
\end{equation}
where we used \eqref{tau1alpha}.
We call the set
\begin{equation}
S:=
\{T(\delta):\delta\in[\alpha,\alpha+2n\pi)\}
\label{biggersolcurve}
\end{equation}
the \emph{solution curve}.
Note that 
\begin{align*}
S&=S_\bR\cap Q,  \text{ where }
\\
Q&  := [0,2\pi)\times [\tau_2(\alpha),\tau_2(\alpha)+2\pi)\times\ldots\times [\tau_n(\alpha),\tau_n(\alpha)+2\pi)
\end{align*}
where we also used \eqref{tau1alpha},
so $[\tau_1(\alpha),\tau_1(\alpha)+2\pi)=[0,2\pi)$.
Geometrically, $S$ can be obtained from $S_0$
with reflections and translations,
while $S_\bR$ can be obtained from $S$ with
translations only.
Another useful property of $S$ is that
for each $\beta\in[0,2\pi)$ 
there is
exactly one $\delta\in[\alpha,\alpha+2n\pi)$
such that $\tau_{1}(\delta)=\beta$.

\begin{figure}[h!]
\begin{center}
\includegraphics[keepaspectratio,height=7cm]{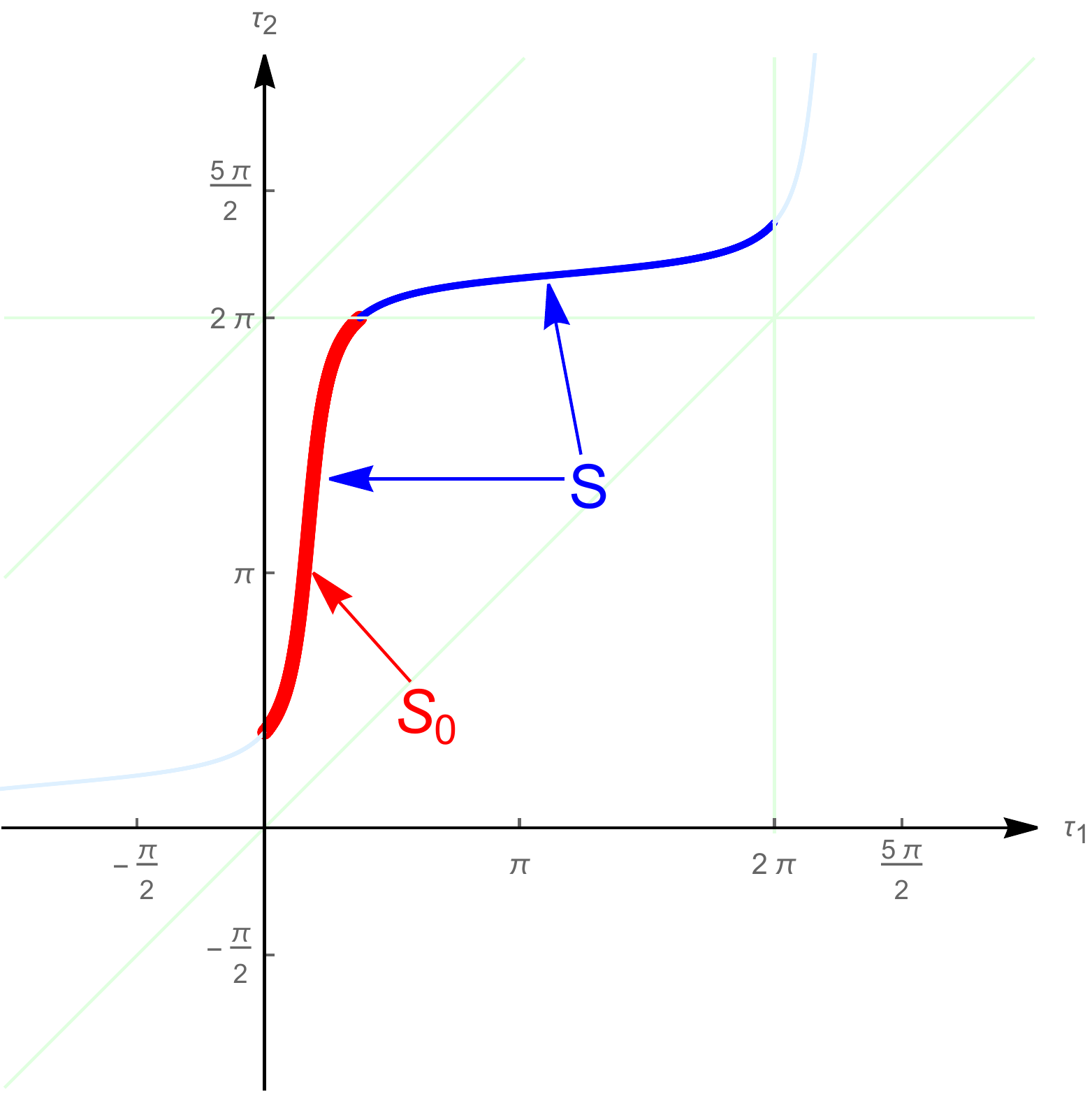}
\quad
\includegraphics[keepaspectratio,height=7cm]{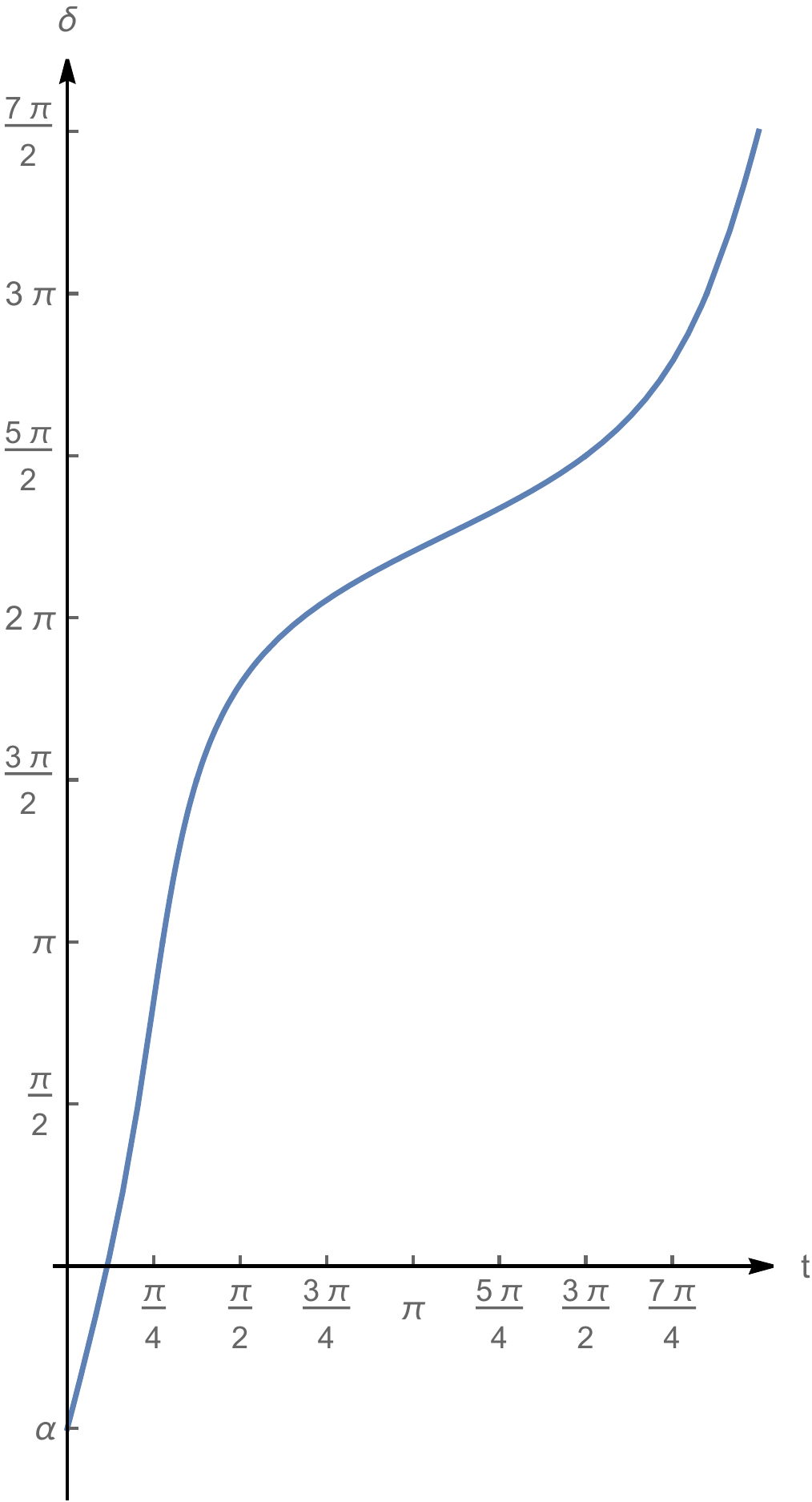}
\end{center}

\caption{Left: solution curve $S$ of
the monic Blaschke product with
zeros at $1/2$ and $(1+i)/2$,
$0\le\tau_1<\tau_2<2\pi$, $B(e^{i\tau_1})=B(e^{i\tau_2})=e^{i\delta}$, $\alpha\le \delta\le \alpha+2\pi$, where $\alpha=-\pi/2$ now.
Right: argument of the same monic Blaschke product, $\delta=\arg B(e^{it})$.}
    \label{fig:solutioncurve}
\end{figure}

These are depicted on the left half of Figure \ref{fig:solutioncurve} where $S_{0}$
is the thick arc 
and it is continued above with 
another arc. 
These two arcs together
form $S$ and describe the motions
of $\tau_1,\tau_2$ together as
$\exp(i\delta)$ goes around the unit circle twice 
($\delta$ grows from $\alpha$ to $\alpha+4\pi$).
Extending these two arcs
with the very thin arcs,
we obtain $S_\bR$, the full solution curve.

Now we recall
the question raised
by Pap and Schipp
in
\cite{PapSchipp2015}.
Consider the pairs of protons, each of charge $+1/2$, at
$\zeta_1,\zeta_1^*,\ldots,\zeta_{n-1},\zeta_{n-1}^*$
as the critical points of a (monic) Blaschke product of degree $n$,
and
the (doubled) discrete energy of electrons restricted to the unit circle
\begin{equation}
W(w_1,\ldots,w_n)
:= \sum_{k=1}^{n-1} \sum_{j=1}^n\log|(w_j-\zeta_k)(w_j-\zeta_k^*)|
-2\sum_{1\leq j<k\leq n}\log|w_j-w_k|
\label{W}
\end{equation}
where $|w_1|=1$, $\ldots$, $|w_n|=1$.
The 
set $S_\bR$ 
connected to the same monic Blaschke product
yields critical
configurations of electrons
for each fixed $\delta$ (which corresponds to fixing one of the electrons),
according to e.g. \cite{PapSchipp2015}.
In other words, for $a_1,\ldots,a_n\in \bD$, using the monic Blaschke product
with zeros at $a_1,\ldots,a_n$
one can construct  pairs of protons as solutions of $B'(z)=0$, and, 
for any given $\delta\in[0,2\pi)$, 
the corresponding  configuration of
electrons as 
all 
solutions of 
$B(z)=e^{i\delta}$.
Then according to the result of Pap and Schipp, Theorem 4 from \cite{PapSchipp2015}, 
these configurations of electrons are critical points of $W$.
The question posed on p. 476 of \cite{PapSchipp2015} is then: 
Are these critical points (local) minima of the
restricted
energy function $\widetilde{W}$
where
$\widetilde{W}(\tau_1,\ldots,\tau_n)
:=
W\left(e^{i\tau_1},\ldots,e^{i\tau_n}\right),
$ 
 $\tau_1\ldots,\tau_n\in \mathbb{R}$?

We give a positive answer to this question in general.
Note that two special cases were solved in \cite{PapSchipp2015}
with
different methods.
Our answer is the following.
There are no other critical points
on the unit circle
(where the
tangential
gradient vanishes).
Moreover, all the points on the
set $S_\bR$ 
are global minimum points of
the restricted energy function $\widetilde{W}$.
 
\begin{theorem}
\label{thm:mainthm1}
Let
$a_1,\ldots,a_n \in
\mathbb{D}$ and
$B(z)$ be the monic Blaschke product 
\eqref{blaschkeprod}
with 
zeros at $a_1,\ldots,a_n$.
Assuming $B'(0)\ne 0$, list up the critical points of $B$ as 
$\zeta_1,\ldots,\zeta_{n-1}\in\bD\setminus \{0\}$ and
$\zeta_1^*,\ldots,\zeta_{n-1}^* \in \bD^*$.

Then
the tangential gradient
of $W$
vanishes on
the points corresponding to the set 
$A\cap Q$
defined in \eqref{Asetdef}
exactly
on the 
set $S$.

More precisely, 
on 
$A\cap Q$,  
it holds that 
$\nabla \,\widetilde{W}(\tau_1,\ldots,\tau_n)=0$
if and only if
$(\tau_1,\ldots,\tau_n)= T(\delta)$
for some 
$\delta\in[\alpha,\alpha+2n\pi)$.

Furthermore,
all points of $S_\bR$ 
are global minimum points of $\widetilde{W}$.

\end{theorem}

Let us recall here a recent result of Simanek \cite[Theorem 2.1]{Simanek}.
Briefly, he established that for any configuration of electrons on the unit circle, there is an external field (collection of protons) such that
the electrons are in electrostatic equilibrium (that is, the gradient of the energy is zero).
We are going to refine this result
by determining the number of pairs of protons and their locations using 
the solution curve defined in \eqref{biggersolcurve}.

For the following we need some 
more 
results on Blaschke products.
Namely for given $z_1,z_2,\ldots,z_n\in\bC$, $|z_j|=1$, $z_j\ne z_k$
($j\ne k$),
we need to find a Blaschke product $B(.)$ of
degree $m$,
such that 
\begin{equation}
B(z_j)=\chi \prod_{k=1}^m \frac{z_j-a_k}{1-\overline{a_k}z_j} = 1,
\quad j=1,2,\ldots,n. 
\label{eq:blaschkeinterpol}
\end{equation}
The first
result of this kind was established by Cantor and Phelps in \cite{CantorPhelps}
(for some $m$)
and the stronger form with degree $m\le n-1$ was given by Jones and Ruscheweyh
in \cite{JonesRuscheweyh}, see also a paper by Hjelle \cite{Hjelle2007}.
By using the results of Jones and Ruscheweyh, Hjelle
showed that there is a
Blaschke product $B(z)$
of degree $m=n$
such that \eqref{eq:blaschkeinterpol}
holds,
see \cite{Hjelle2007}, p. 44.
We will use one such particular Blaschke product
$B(z)=B(z_1,z_2,\ldots,z_n;z)$
corresponding to $z_1,z_2,\ldots,z_n$.
Note that Hjelle's Blaschke product is not unique, 
since there is an extra iterpolation condition.
Observe that the extra interpolation condition can be chosen so that
$B'(0)\ne 0$ is satisfied.

\begin{theorem}
\label{thm:reverseproblem}
For distinct $z_1,\ldots,z_n\in \partial \mathbb{D}$ 
consider the Blaschke product $B(z)$ introduced above.
Assume that $B'(0)\ne 0$.

Then there exist $\zeta_1,\zeta_2,\ldots,\zeta_{n-1}\in \bD\setminus\{0\}$
such that the (doubled) energy function $W$, defined in \eqref{W}
has critical point at $(w_1,\ldots,w_n)=(z_1,\ldots,z_n)$
(even regarded as a point of $\bC^n$).
Moreover, 
on $(\partial \mathbb{D})^n$,
$W|_{(\partial \mathbb{D})^n}$ 
has global minimum
at
$(z_1,\ldots,z_n)$.
\end{theorem}

\section{Some basic propositions}

Recall that it was given in \eqref{W}
the discrete energy of an electron  
configuration $w_1,\ldots,w_n\in\bC$
(with charges $-1$)
in presence of an external field
generated by pairs of fixed protons
$\zeta_1,\zeta_1^*,\zeta_2,\zeta_2^*,\ldots,\zeta_{n-1},\zeta_{n-1}^*$ (with charges $+1/2$ each),
where $\zeta_1,\ldots,\zeta_{n-1}\in\bD$.
Note that actually $W$ is
the double of the physical
energy of the system (see also
\cite{ETNA2005}, p. 22 where they
use this form of
discrete energy).
We will see later on why it is more
convenient to use this "doubled energy".

Sometimes the following exceptional set will be excluded:
\begin{multline}
\label{diszkretkiveteleshalmaz}
 E:= \big\{(w_1,\ldots,w_n,\zeta_1,\ldots,\zeta_{n-1}) \in \bC^{n}\times \bD^{n-1}: \ \\
\zeta_j= 0 \mbox{ for some } j
\mbox{ or }
w_j= w_k \mbox{ for some } j\ne k,
\\
\mbox{ or }
\zeta_j = w_k
\mbox{ or } \zeta_j^* = w_k
\mbox{ for some } j,k
 \big\}.
\end{multline}
This is a closed set with empty interior.
Geometrically, this set covers the cases
when some of the protons are at the origin,
some of the electrons are at the same position
or a proton and an electron are at the same position.
 Let us remark also that
$W=W(w_1,\ldots,w_n)$
is locally
the
real part
of a holomorphic
function
when $\zeta_1,\ldots,\zeta_{n-1}$
are fixed and $W$ is considered
on $(w_1,\ldots,w_n)\in \bC^n$ such that $(w_1,\ldots,w_n,\zeta_1,\ldots,\zeta_{n-1}) \not\in E$.

\medskip

This energy can be generalized
substantially.
Let $\mu$ be a signed measure on $\bC$.
We define the
(doubled)
energy in this case as
\begin{gather}
W_{\mu,1} :=2\sum_{k=1}^n
\int_{\bC} \log|w_k-\zeta| d\mu(\zeta),
\quad
W_{\mu,2} :=
\sum_{\substack{l \ne k\\ 1\le l,k\le n}}
\log|w_l-w_k|
, \text{ and }  \notag
\\
W_\mu(w_1,\ldots,w_n)  :=
W_{\mu,1} - W_{\mu,2}.
\label{Wmu}
\end{gather}
Note that in \eqref{W} we sum over all $l< k$
pairs and there is an extra factor $2$.
In \eqref{Wmu}, the sum is over all $l\ne k$ pairs. Later this second, symmetric expression will be more convenient.

Here, it may happen that $W_{\mu,1}$ or $W_{\mu,2}$ becomes infinity, so we again
introduce the exceptional
set as follows:
\begin{multline}
E_\mu:=
\{\left(w_1,\ldots,w_n\right)\in\bC^n:
\
w_j=w_k \mbox{ for some }
j\ne k
\\
\mbox{ or }
\int_{\bC} \left|\log| w_j-\zeta| \right|d|\mu|(\zeta) = +\infty
\mbox{ for some }
j
\}.
\end{multline}
Note that finiteness of this latter
integral is equivalent to
the finiteness of the potentials of $\mu_+$ and $\mu_-$ at $w_j$
where $\mu_+$, $\mu_-$ are the positive and negative parts of $\mu$ respectively.
Observe that if $(w_1,\ldots,w_n)\not\in E_\mu$, then $W_{\mu,1}$
and $W_{\mu,2}$ are finite, and so is $W_\mu$.

\medskip

An important tool in our investigations is the
Cayley transform and its inverse.
Basically,
it is just a transformation
between a half-plane and
the unit disk, though there is no widely accepted, standard form of it.
We use the following form, which we call inverse Cayley transform
\begin{equation*}
C(z)=C_\theta(z):=
i \frac{1+z e^{-i\theta}}{1-ze^{-i\theta}}
\end{equation*}
 where $\theta\in\bR$ will be specified later.
It is standard to verify that
$C(z)$ maps the unit disk onto the upper half-plane,
$C_\theta(e^{i\theta})=\infty$,
and $C(.)$ maps
bijectively
the unit circle (excluding $e^{i\theta}$)
to the real axis.
Furthermore, $C_\theta(e^{it})$ is
continuous and strictly increasing from $t=\theta$ to $t=\theta+2\pi$,
$C_\theta(e^{it})\rightarrow -\infty$ as $t \rightarrow \theta +0$,
$C_\theta(e^{it})\rightarrow +\infty$ as $t \rightarrow \theta+2\pi -0$.
It is easy to see that $C(z^*)=\overline{C(z)}$
and $C'(z)\ne 0$ (if $z\ne e^{i\theta}$).
Later we will use the
Cayley transform  too:
\[
C_\theta^{-1}(u)=e^{i\theta} \frac{u-i}{u+i}.
\]

Mapping the electrons and protons
by
$C_\theta$,
we define $t_j$ with $t_j=C_\theta(w_j)$.
We  also write $\xi_j:=C_\theta(\zeta_j)$ and accordingly, $\overline{\xi_j}=C_\theta(\zeta_j^*)$
and investigate
the following
new discrete energy:
\begin{equation} \label{V}
V(t_1,\ldots,t_n):=\sum_{k=1}^{n-1} \sum_{j=1}^n\log|(t_j-\xi_k)(t_j-\overline{\xi}_k)|-2\sum_{1\leq j<k\leq n}\log|t_j-t_k|.
\end{equation}

We also define
the (doubled) 
discrete energy  on the real line when the external field is determined by a signed measure $\nu$:
\begin{gather}
V_{\nu,1}  :=
2\sum_{k=1}^n
\int_{\bC} \log|t_k-\xi| d\nu(\xi),
\quad
V_{\nu,2} :=
\sum_{\substack{l \ne k\\ 1\le l,k\le n}}
\log|t_l-t_k| \text{ and }
\notag
\\
V_\nu(t_1,\ldots,t_n):= V_{\nu,1} - V_{\nu,2}.
\end{gather}

We introduce again
the exceptional set corresponding to $\nu$ as follows:
\begin{multline*}
E_\nu:=\{
\left(t_1,\ldots,t_n\right)\in\bC^n:
\
t_j=t_k \mbox{ for some }
j\ne k
\\
\mbox{ or }
\int_{\bC} \left|\log| t_j-\xi| \right|d|\nu|(\xi) = +\infty
\mbox{ for some }
j
\}.
\end{multline*}

\medskip

The next result gives a somewhat surprising connection
how the
 inverse
Cayley transform carries over energy.
Actually, there is a cancellation in the background
which makes it work.

\begin{proposition}
\label{contenergtrfall}
Fix $\theta\in\bR$ and let
$\mu$ be a signed measure
on $\bC$
with compact support
such that
$\mu(\{0\})=0$,
$\mu(\bC)=n-1$.
Write $\nu:=\mu\circ C_\theta^{-1}$, that is,
$\nu(B)=\mu(C_\theta^{-1}(B))$
for every Borel set $B$.

Assume that $w_1,\ldots,w_n\in\bC$
and
$(w_1,\ldots,w_n)\not\in E_\mu$
and
\begin{equation}
 \int_\bC \log|\zeta-e^{i\theta}| d\mu(\zeta) \text{ is finite.}
\label{cnuvegesfelt-zeta}
\end{equation}

Then
with
$t_1,\ldots,t_n\in\bC$ where $t_j=C_\theta(w_j)$,
we know that
$(t_1,\ldots,t_n)\not\in E_\nu$,
$W_\mu(w_1,\ldots,w_n)$
and
$V_\nu(t_1,\ldots,t_n)$ are finite and
we can write
\begin{equation}
\label{magickanccont}
 W_\mu\left(w_1,\ldots,w_n\right)
 =
 V_\nu\left(t_1,\ldots,t_n\right)
 +c
\end{equation}
where $c$ is a finite constant, namely
\begin{equation} \label{c_konst}
c=n(n-1)\log(2)
-2n \int_\bC \log|\xi+i|d\nu(\xi).
\end{equation}

\end{proposition}

\begin{proof}
It is straightforward to verify that
$(t_1,t_2,\ldots,t_n)\not\in E_\nu$.
Furthermore,
\begin{multline*}
 \int_\bC \log|\xi+i| d\nu(\xi)
 =
\int_\bC \log|C_\theta(\zeta)+i| d\mu(\zeta)
=
\int_\bC \log\left|
i
\left(1+
\frac{1+\zeta e^{-i\theta}}
{1-\zeta e^{i\theta}}
\right)
\right| d\mu(\zeta)
\\
=\int_\bC \log(2)-\log|\zeta-e^{i\theta}| d\mu(\zeta),
\end{multline*}
so \eqref{cnuvegesfelt-zeta} is equivalent to
\begin{equation}
 \int_\bC \log|\xi+i| d\nu(\xi) \text{ is finite.}
\label{cnuvegesfelt}
\end{equation}
Note that this entails the finiteness of $c$ defined in \eqref{c_konst}.

With the notation of the Proposition,
\begin{multline}
\label{magicproof}
W_\mu\left(w_1,\ldots,w_n\right)
-V_\nu\left(t_1,\ldots,t_n\right)
=
2\sum_{k=1}^n \int_\bC \log|w_k-\zeta| d\mu(\zeta)
\\
- \sum_{\substack{j \ne k\\ 1\le j,k\le n}} \log|w_j-w_k|
- 2\sum_{k=1}^n \int_\bC \log|t_k-\xi| d\nu(\xi)
+
\sum_{\substack{j \ne k\\ 1\le j,k\le n}} \log|t_j-t_k|
\end{multline}
where we investigate the difference of the integrals and difference of the sums separately.
So we write
\begin{multline*}
 \int_\bC \log|w_k-\zeta| d\mu(\zeta)
 -
 \int_\bC \log|t_k-\xi| d\nu(\xi)
 \\
 =\int_\bC \log|C_\theta ^{-1}(t_k)-C_\theta ^{-1}(\xi)| d\nu(\xi)
 -\int_\bC \log|t_k-\xi| d\nu(\xi)
\\ =
\int_\bC \log\left|e^{i\theta} \left(\frac{t_k-i}{t_k+i}-\frac{\xi-i}{\xi+i}\right) \right| -\log\left|t_k-\xi\right|
d\nu(\xi)
\\=
\int_\bC \log(2) +\log\left|\frac{1}{(t_k+i)(\xi+i)}\right|
d\nu(\xi)
\\=
\int_\bC -\log|\xi+i| d\nu(\xi)
+\left(\log(2)-\log|t_k+i|\right) \nu(\bC),
\end{multline*}
where this last integral exists, by assumption \eqref{cnuvegesfelt}.
Similarly,
\begin{multline*}
    \log|t_j-t_k|-\log|w_j-w_k|=\log|t_j-t_k|-\log|C_\theta ^{-1}(t_j)-C_\theta ^{-1}(t_k)|
    \\=
    \log|t_j-t_k|-\log\left|e^{i\theta} \left(\frac{t_j-i}{t_j+i}\right)-e^{i\theta} \left(\frac{t_k-i}{t_k+i}\right)\right|
    \\=
    -\log(2)+
    \log|t_j+i|+\log|t_k+i|.
\end{multline*}

Substituting into \eqref{magicproof},
we get
\begin{align*}
    W_\mu \left(w_1,\ldots,w_n\right)
& -V_\nu\left(t_1,\ldots,t_n\right)
\\& =
2\sum_{k=1}^n \left(\int_\bC -\log|\xi+i| d\nu(\xi)
+\left(\log(2)-\log|t_k+i|\right) \nu(\bC) \right)
\\
& \qquad \qquad +\sum_{\substack{j \ne k\\ 1\le j,k\le n}} \left( -\log(2)+\log|t_j+i|+\log|t_k+i| \right)
\\& =
-2\nu(\bC)\sum_{k=1}^n \log|t_k+i|+
2n\nu(\bC)\log(2)-
2 n \int_{\bC} \log|\xi+i| d\nu(\xi)
\\ & \qquad \qquad
-n(n-1)\log(2) +
2(n-1)\sum_{k=1}^n \log|t_k+i|
\\& =
n(n-1)\log(2)
-2n \int_\bC \log|\xi+i|d\nu(\xi),
\end{align*}
where we used that
$\nu(\bC)=n-1$.
\end{proof}

\begin{remark}
Since $\mu$ has compact support,
$\supp\nu$  is disjoint from $-i$,
moreover, their distance is positive.
Hence the logarithm in the integral in \eqref{c_konst}
is bounded from below.
It is not necessarily bounded from above,
but we assume  \eqref{cnuvegesfelt} directly.
Instead of supposing \eqref{cnuvegesfelt},
we may suppose
that $\mu$ and $\theta$ (from Cayley transform) are such that $\supp\mu$ and $e^{i\theta}$
are of positive distances from each other.
This would ensure that
$\supp \nu$ remains bounded
entailing that
the logarithm in the integral in \eqref{cnuvegesfelt}
is bounded from above.
In other words, if $\supp \mu$ is compact and $e^{i\theta}\not\in \supp\mu$, then \eqref{cnuvegesfelt} holds.
\end{remark}

We note that
this Proposition \ref{contenergtrfall} extends the result of Theorem 6 in Pap, Schipp \cite{PapSchipp2015}
that we allow arbitrary signed external fields
in place of discrete protons located symmetrically with respect to the unit circle.

\begin{proposition}
\label{energtrfalllim}
We maintain the assumptions and notations of Proposition \ref{contenergtrfall}.
Let $\ell\in\{1,\ldots,n\}$ and let $w_j$, $j\ne \ell$ be fixed.

Assume that
\begin{equation}
e^{i\theta} \not\in\supp\mu
\label{hatarpontnagyonszep}
\end{equation}
and assume further that
replacing $w_\ell$ by $e^{i\theta}$, we have
\begin{equation}
    (w_1,\ldots,e^{i\theta},\ldots,w_n)\not\in E_\mu. \label{hatarpontszep}
\end{equation}
If $w_\ell\rightarrow e^{i\theta}$,
then $|t_\ell|=|C_\theta(w_\ell)|\rightarrow\infty$
and we get that
\begin{equation}
 W_\mu(w_1,\ldots,w_{\ell-1},e^{i\theta},w_{\ell+1},\ldots,w_n)
 = V_\nu(t_1,\ldots,t_{\ell-1},\infty,t_{\ell+1},\ldots,t_n) +c
 \label{magickcanclim}
\end{equation}
where $c$ is the constant defined in \eqref{c_konst} and
\begin{multline} \label{Vhianyos}
 V_\nu(t_1,\ldots,t_{\ell-1},\infty,t_{\ell+1},\ldots,t_n)
 :=  V_\nu(t_1,\ldots,t_{\ell-1},t_{\ell+1},\ldots,t_n)
 \\ = 2\sum_{\substack{j=1\\j\ne\ell}}^n
 \int_\bC \log|t_j-\xi| d\nu(\xi)
  -\sum_{\substack{1\leq j,k\leq n\\j\ne \ell,k\ne \ell,j\ne k}}\log|t_j-t_k|.
\end{multline}
\end{proposition}
\begin{proof}
First, we discuss why the integrals appearing here are finite.
By slightly abusing the notation,
$W_\mu(w_\ell):=W_\mu(w_1,\ldots,w_\ell,\ldots,w_n)$ is finite at $w_\ell=e^{i\theta}$, because of \eqref{hatarpontszep}.
Assumption \eqref{hatarpontnagyonszep} implies
that there is a neighborhood $U$ of $e^{i\theta}$
such that its closure $U^-$  is disjoint from $\supp\mu$,
$U^-\cap \supp\mu=\emptyset$.
Therefore $W_\mu(w)$
 is also finite when $w\in U$,
moreover $W_\mu(.)$ is continuous there.
Similarly, we use $V_\nu(t):=V_\nu(t_1,\ldots,t_{\ell-1},t,t_{\ell+1},\ldots,t_n)$ (abusing the notation again).
Obviously, $C_\theta(U)$ is an unbounded open set on
the extended complex plane $\bC_\infty$
and is a neighborhood of infinity.
By Proposition \ref{contenergtrfall},
 $V_\nu(t)$ is defined on  $C_\theta(U)\setminus\{\infty\}$, has finite value and is
 continuous there.
Moreover, $V_\nu(t)$ has finite limit as $t\rightarrow\infty$.
By \eqref{hatarpontszep} and \eqref{hatarpontnagyonszep},
$(w_1,\ldots,w_{\ell-1},w,w_{\ell+1},\ldots,w_n)\not \in E_\mu$ for $w\in U$.
Hence $(t_1,\ldots,t_{\ell-1},t,t_{\ell+1},\ldots,t_n)\not \in E_\nu$ for $t\in C_\theta(U)\setminus\{\infty\}$.
This also implies that $\int_\bC\log|t_j-\xi| d\nu(\xi)$ is finite, $j=1,\ldots,n$, $j\ne \ell$,
which are the integrals appearing on the right of \eqref{Vhianyos}.

Regarding $V_\nu$, we write
\begin{multline*}
\lim_{t_\ell \rightarrow \infty} V_\nu(t_\ell)
=
\lim_{t_\ell \rightarrow \infty}\left(  2\sum_{j=1}^n
\int_\bC \log|t_j-\xi| d\nu(\xi)
    -\sum_{\substack{1\leq j,k\leq n \\ j\ne k}}\log|t_j-t_k|\right)
\\=
2\sum_{\substack{j=1\\j\neq \ell}}^n \int_\bC \log|t_j-\xi| d\nu(\xi)
    -\sum_{\substack{1\leq j, k \leq n \\j\neq \ell, k\neq \ell}}
    \log|t_j-t_k|
\\
+ \lim_{t_\ell \rightarrow \infty}
    \left(
    2\int_{\bC} \log|t_\ell-\xi| d\nu(\xi)
    -\sum_{\substack{1 \leq j, k \leq n \\ k \neq j, k=\ell \ \text{or} \  j=\ell} }
    \log|t_j-t_k| \right)
    \\=
    V(t_1,\ldots,t_{\ell-1},t_{\ell+1},\ldots,t_n),
\end{multline*}
where in the last step we used the following calculation.
\begin{multline*}
    \lim_{t_\ell \rightarrow \infty}\left(
    2\int_\bC \log|t_\ell-\xi| d\nu(\xi)
    -\sum_{\substack{1 \leq j, k \leq n \\ k \neq j, k=\ell \ \text{or} \  j=\ell}}
    \log|t_j-t_k| \right)
\\
= \lim_{t_\ell \rightarrow\infty} 2 \int_\bC \log|t_\ell|
+\log\left|1-\frac{\xi}{t_\ell}\right| d\nu(\xi)
-
2 \sum_{\substack{1 \leq j \leq n \\ j\ne \ell} } {
\left( \log|t_\ell| + \log\left|1-\frac{t_j}{t_\ell}\right| \right)}
\end{multline*}
where $\int_\bC \log|t_\ell| d\nu(\xi)=(n-1) \log|t_\ell|$
so the first term in the integral and in the sum cancel each other, by $\nu(\bC)=n-1$.
Regarding the second term in the sum, it tends to zero.
The second term in the integral also tends to zero, because the support of $\nu$ is compact, hence $\log|1+\xi/t_\ell|$ tends to $0$ uniformly.

Using this calculation,
\eqref{magickanccont} from Proposition \ref{contenergtrfall}
and the properties of $W_\mu$
and $C_\theta$ we get that
\begin{multline*}
W_\mu(e^{i\theta})
=
\lim_{w_\ell \to e^{i\theta}} W_\mu(w_\ell)
\\=
\lim_{t_\ell \to \infty}
\left(V_\nu(t_\ell)+c\right)
=
V_\nu(t_1,\ldots,t_{\ell-1},t_{\ell+1},\ldots,t_n)+c.
\end{multline*}
\end{proof}

Based on the above proposition, it is justified to extend the definition of $V_\nu$ by continuity as
$V_\nu(t_1,\ldots,t_{\ell-1},\infty,t_{\ell+1},\ldots,t_n):=V_\nu(t_1,\ldots,t_{\ell-1},t_{\ell+1},\ldots,t_n)$
 in case $t_\ell$
becomes $\pm \infty$.

Now we are going to relate the critical points of
$W_\mu$ and $V_\nu$ when the configurations of the electrons are restricted to the unit circle (or to the real line).

When the electrons are restricted
to the unit circle, that is,
\begin{equation}
|w_j|=1, \qquad j=1,\ldots,n
\label{electronsonthecircle}
\end{equation}
we are going to introduce the
tangential
gradient as follows.
In this case, in addition to supposing that $\mu$ has compact support, we assume that $\supp \mu$ is disjoint from the unit circle.

We write
\begin{gather}
w_j=e^{i\tau_j}, \quad j=1,\ldots,n,
\label{wjtauj}
\qquad
\widetilde{W}_\mu(\tau_1,\ldots,\tau_n)
:=
W_\mu\left(e^{i\tau_1},\ldots,e^{i\tau_n}\right).
\end{gather}

We call $\nabla \widetilde{W}_\mu$ the
tangential
gradient of $W_\mu$.
$\nabla \widetilde{W}_\mu$ of $\widetilde{W_\mu}$
has special meaning with respect to
the complex derivative of $W_\mu$:
it is the tangential component  of $\nabla W_\mu$ with respect to the unit circle.
Similar distinction also appears in \cite{Simanek},
see the definitions of $\Gamma$-normal
electrostatic equilibrium and
total electrostatic equilibrium on p.~2255.
This total electrostatic equilibrium appears in Theorem 2, \cite{PapSchipp2001} which will be used later.

\begin{proposition}
\label{szimmenergiavalos}
Let $\nu$ be a signed measure on $\bC$
with compact support.
Assume that $\supp \nu$ is disjoint from the real line and $\nu$ is symmetric with respect to
the
real line: $\nu(H)=\nu(\overline{H})$
where $H\subset \{\Im(u)>0\}$ is a Borel set and $\overline{H}=\{\overline{u}:\ u\in H\}$ denotes the complex conjugate.

Then for $u_1,\ldots,u_n\in\bR$ we have for the $j$-th imaginary directional derivative
(with direction $i{\bf e}_j:= i (0,\ldots,0,1,0,\ldots,0)$)
that
\begin{multline}
\partial_{i{\bf e}_j} V_\nu(u_1,\ldots, u_n)
\\:=
\lim_{v_j\rightarrow 0}
\frac{V_\nu(u_1,\ldots,u_j +i v_j,\ldots,u_n)-V_\nu(u_1,\ldots, u_n)}{v_j}
=0.
\end{multline}
\end{proposition}
Roughly speaking, if the external field is symmetric, then the forces moving the electrons will keep the
electrons on the real line (all coordinates of gradient are parallel with the real line).

\begin{proposition}
\label{szimmenergiakorvonal}
Let $\mu$ be a signed measure on $\bC$
with compact support.
Assume that $\supp \mu$ is disjoint from the unit circle
and $\mu$ is symmetric with respect to
the
unit circle: $\mu(H)=\mu(H^*)$
where $H\subset \{|w|<1\}$ is a Borel set and $H^*=\{1/\overline{w}:\ w\in H\}$ denotes the inversion of $H$.

Then for $|w_1|=\ldots=|w_n|=1$, we have for
the $j$-th normal derivative (with direction $w_j{\bf e}_j$) that
\begin{multline}
\partial_{w_j{\bf e}_j} W_\mu(w_1,\ldots, w_n)
\\:=
\lim_{\varepsilon\rightarrow 0}
\frac{W_\mu(w_1,\ldots,w_j +\varepsilon w_j,\ldots,w_n)-W_\mu(w_1,\ldots, w_n)}{\varepsilon}
=0.
\end{multline}
\end{proposition}

Note that because $\mu$ has compact support and is symmetric with respect to the unit circle,
we necessarily have that $0$ is not in $\supp\mu$.

Roughly speaking,
Proposition \ref{szimmenergiakorvonal} states
that if the measure $\mu$ is symmetric with respect to the unit circle,
then the gradient and the
tangential 
gradient of $W_\mu$ are the same.
In other words,
$n$ electrons on the unit circle, allowed to move freely on the plane in the external field generated by $\mu$ will stay on the unit circle.

\begin{proof}[Proofs of Propositions \ref{szimmenergiavalos} and \ref{szimmenergiakorvonal}]
To see Proposition \ref{szimmenergiavalos},
we fix $u_1,\ldots,u_{j-1}$, $u_j$,
$u_{j+1},\ldots, u_n\in \bR$,
and use here $J(.)$ for the conjugation: $J(u)=\overline{u}$.
Writing
$V(u) :=V_\nu(u_1,\ldots,u_{j-1},u,u_{j+1},\ldots,u_n)$ for general complex $u=u_j+iv_j$, and using that
$\nu$ is symmetric to the real line, in other words,
$\nu(H)=\nu(J(H))$ for Borel sets $H$, we find
\[V(u_1,\ldots,u_{j-1}, u, u_{j+1},\ldots, u_n)=V(u_1,\ldots,u_{j-1}, J(u), u_{j+1},\ldots, u_n).\]
Therefore,
\begin{align*}
& \partial_{i {\bf e}_j}V(u_1,\ldots,u_{j-1}, u_j, u_{j+1},\ldots, u_n)
\\ &= \frac{\partial V(u_1,\ldots,u_{j-1}, u_j+iv_j, u_{j+1},\ldots, u_n)}{\partial v_j}|_{(u_1,\ldots,u_{j-1}, u_j, u_{j+1},\ldots, u_n)}
\\ &=\frac{\partial V(u_1,\ldots,u_{j-1}, u_j-iv_j, u_{j+1},\ldots, u_n)}{\partial v_j}|_{(u_1,\ldots,u_{j-1}, u_j, u_{j+1},\ldots, u_n)}
\\ &=\frac{\partial V(u_1,\ldots,u_{j-1}, u_j+iv_j, u_{j+1},\ldots, u_n)}{\partial (-v_j)}|_{(u_1,\ldots,u_{j-1}, u_j, u_{j+1},\ldots, u_n)}
\\ & = - \partial_{i {\bf e}_j} V (u_1,\ldots,u_{j-1}, u_j, u_{j+1},\ldots, u_n)
\end{align*}
showing that Proposition \ref{szimmenergiavalos}
holds.

To see Proposition \ref{szimmenergiakorvonal},
we use that the
inverse Cayley transform
is a conformal mapping,
hence it is locally orthogonal.
\end{proof}

\section{ The case of finitely many pairs of protons}

In this section, we specialize the
propositions
of the previous section.
Most of the results here simply follow from those
statements.

We consider the case when $\supp \mu$ is a finite set with $2n-2$ elements, which are symmetric with respect to
the
unit circle and the support is disjoint from the unit circle and the origin:
\begin{gather*}
\supp\mu=\{\zeta_1,\zeta_2,\ldots,\zeta_{n-1},
\zeta_1^*,\zeta_2^*,\ldots,\zeta_{n-1}^*\},
\\
0<|\zeta_j|<1,
\
\mu(\{\zeta_j\})=\mu(\{\zeta_j^*\})=1/2, \quad j=1,2,\ldots,n-1,
\\
\zeta_j\ne\zeta_k,
\quad j,k=1,2,\ldots,n-1,\  
j \ne k.
\end{gather*}
Recall that $\zeta^*=1/\overline{\zeta}$.

The restriction $\zeta_j\ne 0$ is essential for the following reasons.
Although $0^*=\infty$ may be introduced,
definition of discrete energy $W$ cannot be meaningfully
defined.
Note that the usefulness of symmetrization of external fields lies in that
the normal component  of
the field generated by the
symmetrized proton configuration identically vanishes on the unit circle.
However,
when there is a proton at the origin,
there is no complementing  system of
protons  $\omega_1,\ldots,\omega_m$
(for no $m$) such that
the total system
$\{\zeta_1,\ldots,\zeta_n,\omega_1,\ldots,\omega_m\}$
would generate a field with identically vanishing normal component on the
unit circle.

Furthermore, the protons at
the origin contribute to the electrostatic field of all protons only with identically zero tangential component all over the unit circle.
Therefore,
studying equilibrium and energy minima on the circle,
protons at the origin
have no contribution, hence can be dropped from the configuration.
However, then the total charge of the system will drop below $-1$. There are results in this essentially different case, too,
see e.g.~\cite{Grunbaum} or \cite{ForresterRogers}, Theorem 4.1 but those necessarily involve assumptions on locations of electrons.

The below Proposition \ref{energtrfall} follows directly
from the more general Proposition \ref{contenergtrfall}.
Roughly speaking, it expresses
how the energy functions are mapped
to one another via the
 inverse Cayley transform in this special case.
We use here the exceptional set $E$ introduced in \eqref{diszkretkiveteleshalmaz}.

\begin{proposition}
\label{energtrfall}
Fix $\theta\in\bR$ and let $\zeta_j\in \bD$, $j=1,\ldots,n-1$.
Consider the parameters
$\zeta_j, \zeta_j^*$
as well as the parameters $\xi_j=C_\theta(\zeta_j)$,
$\overline{\xi_j}=C_\theta(\zeta_j^*)$.

Assume that $w_1,\ldots,w_n\in\bC$
are such that
$(w_1,\ldots,w_n,\zeta_1,\ldots,\zeta_{n-1})\not\in E$,
and
$w_j\ne e^{i\theta}$ ($j=1,\ldots,n$).

With
$t_1,\ldots,t_n\in\bC$ where $t_j=C_\theta(w_j)$,
we can write
\begin{equation}
\label{magickanc}
 W\left(w_1,\ldots,w_n\right)
 =
 V\left(t_1,\ldots,t_n\right)
 +c
\end{equation}
where $c$ is a constant,
\begin{equation}
c=n(n-1)\log(2)-n\sum_{k=1}^{n-1}\log|(\xi_k+i)(\overline{\xi}_k+i)|.
\label{c_konst_diszkr}
\end{equation}

If $(w_1,\ldots,w_n,\zeta_1,\ldots,\zeta_{n-1})\in E$,
then $W$, $V$ or $c$ is infinite.
\end{proposition}

Next we formulate the following special case of  Proposition \ref{energtrfalllim}.

\begin{proposition}
\label{energtrfalllimdiscr}

Let $\ell\in\{1,\ldots,n\}$ and let $w_j$,
$j\ne \ell$ be fixed such that $w_j\ne e^{i\theta}$ for all $j \ne \ell$.
If $w_\ell = e^{i\theta}$, then $t_\ell=C_\theta(w_\ell)=\infty$
and we get that
\begin{equation}
 W(w_1,\ldots,w_{\ell-1},e^{i\theta},w_{\ell+1},\ldots,w_n)
 =
 V(t_1,\ldots,t_{\ell-1},\infty,t_{\ell+1},\ldots,t_n) +c
 \label{magickcanclimcont}
\end{equation}
where $c$ is defined in \eqref{c_konst_diszkr}
and similarly to \eqref{Vhianyos}
\begin{multline}
 V(t_1,\ldots,t_{\ell-1},\infty,t_{\ell+1},\ldots,t_n):=
  V(t_1,\ldots,t_{\ell-1},t_{\ell+1},\ldots,t_n)
\\
=\sum_{k=1}^{n-1} \sum_{\substack{j=1\\j\ne\ell}}^n\log|(t_j-\xi_k)(t_j-\overline{\xi}_k)|
-2\sum_{\substack{1\leq j<k\leq n\\j\ne \ell,k\ne \ell}}\log|t_j-t_k|.
\end{multline}
\end{proposition}

In Figure \ref{fig:confgandcayley}, particular sets of electrons and protons are shown along with the transformed configuration on the real axis.
Namely, the zeros of the monic Blaschke product $B(.)$
are
$1/2$, $(1+i)/2$, $2/3 i$, $-3/4 i$ and $-7/10+ 6/10 i$.
The protons are at the critical points of this monic Blaschke product $B'(.)=0$ : $0.38-2.21 i$, $1.69+1.13 i$,
$0.68+1.86 i$, $-0.99+0.94 i$, $-0.53+0.51 i$
, $0.17+0.47 i$, $0.41+0.27 i$, $0.08-0.44 i$
(here and in the remaining part of this paragraph
the numbers are rounded to two decimal digits).
The electrons are at the solutions of $B(.)=1$,
and their arguments are: $-2.87$, $-1.19$, $0.41$, $1.28$, $2.33$.
For the inverse Cayley transform, $\theta=-2.87$, that is,
the first electron is mapped to infinity.

\begin{figure}[h]
\noindent
\begin{center}
\begin{minipage}{0.35\textwidth}
\includegraphics[keepaspectratio,height=50mm]{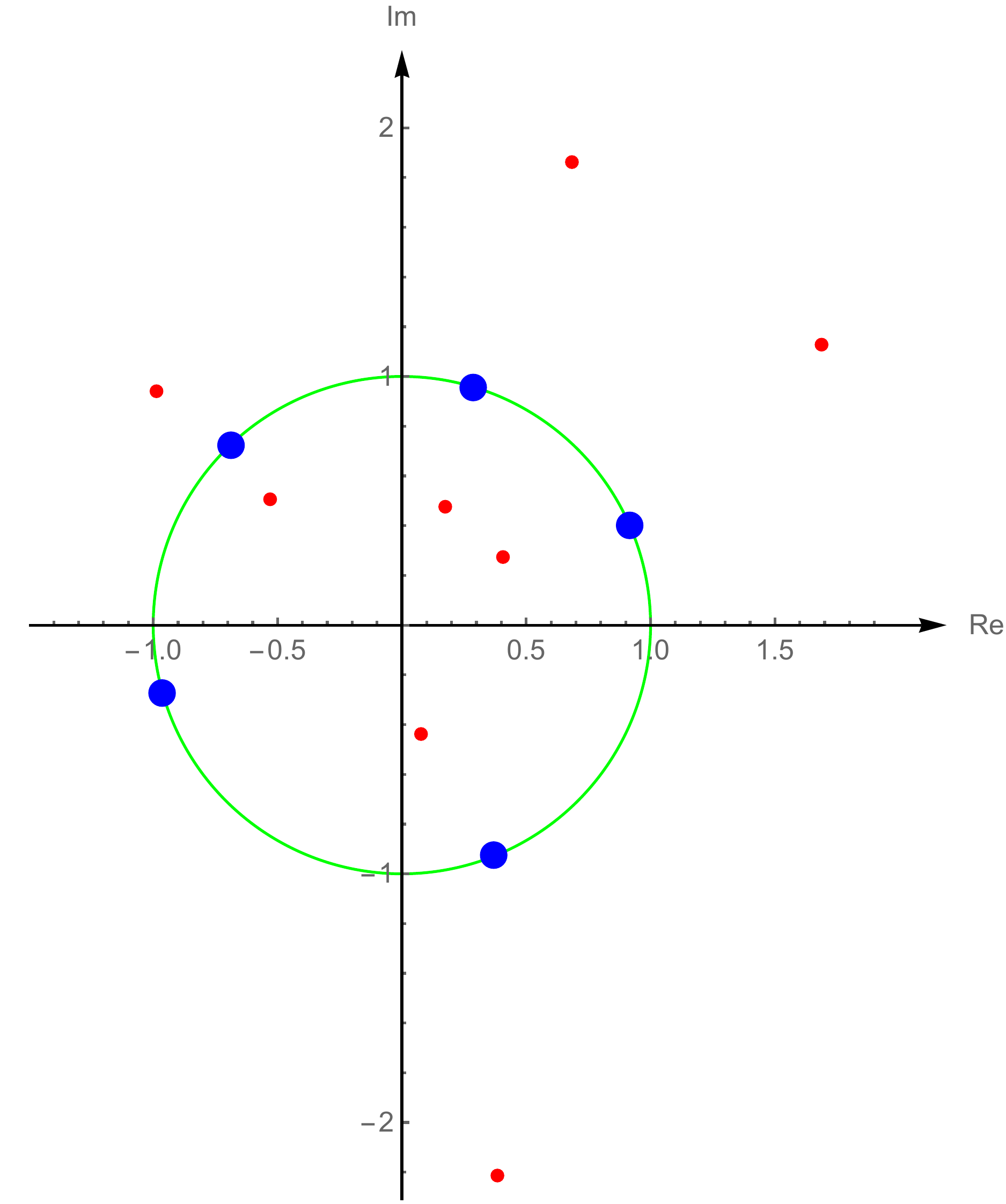}
\end{minipage}
\begin{minipage}{0.35\textwidth}
\includegraphics[keepaspectratio,height=50mm]{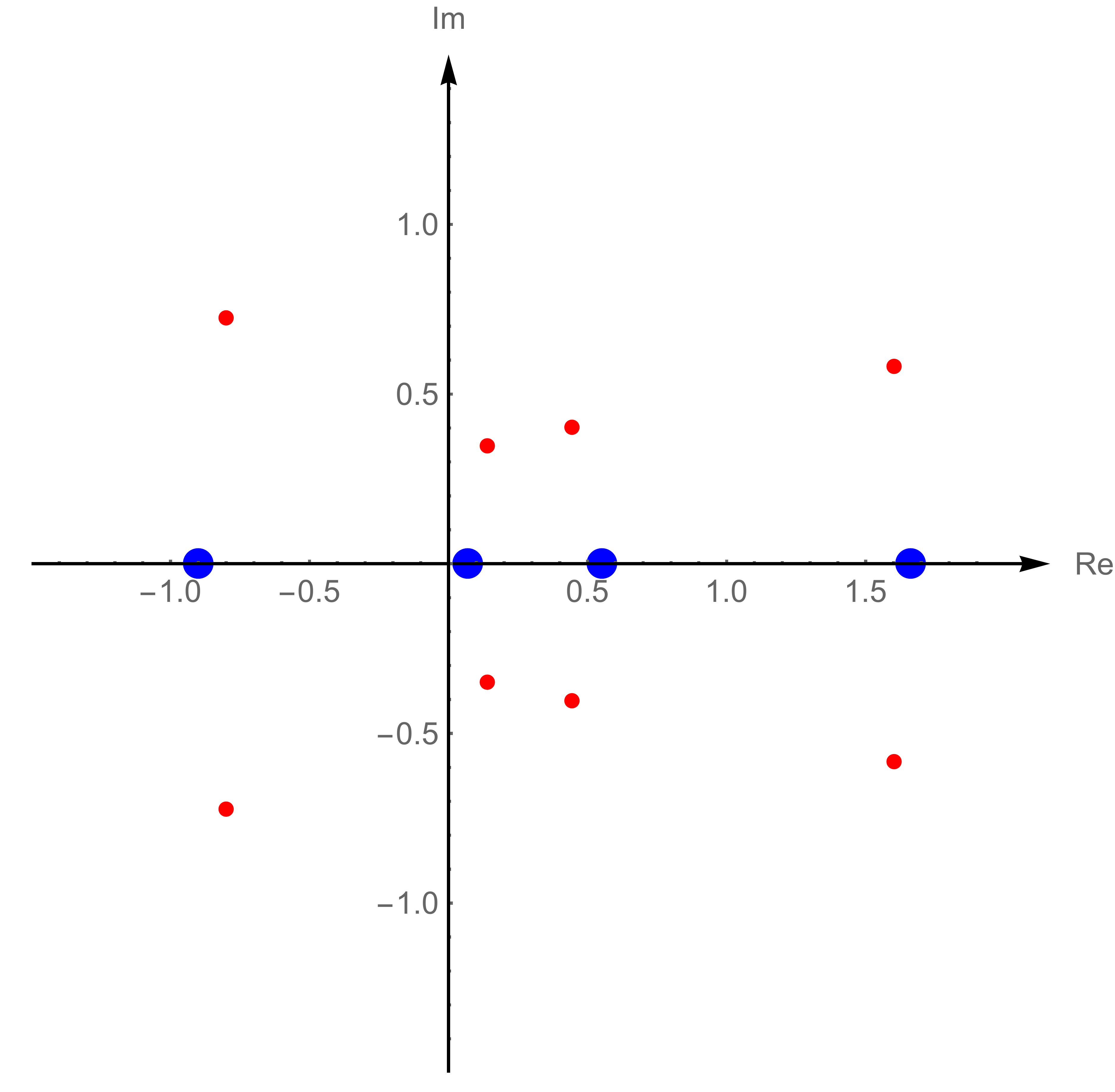}
\end{minipage}
\caption{Equilibrium configurations of five electrons on the unit circle and
the transformed configuration, with one electron transferred to $\infty$.}
\label{fig:confgandcayley}
\end{center}
\end{figure}

In the next proposition we point out, 
how the critical points of the original and the transformed energy function correspond to each other.

\begin{proposition} \label{pr:criticpoints}
Let $\zeta_j\in \bD$, $j=1,\ldots,n-1$
and $w_j\in \bC$, $j=1,\ldots,n$.
Assume that $w_j$'s are restricted to the unit circle, i.e. \eqref{electronsonthecircle}
and \eqref{wjtauj} hold.
We also assume that
$(w_1,\ldots,w_n,\zeta_1,\ldots,\zeta_{n-1})\not\in E$.

Fix $w_1$ and $\tau_1\in\bR$ and
assume that
$(\tau_1,\tau_2,\ldots,\tau_n)\in A$.
Consider the inverse Cayley mapping $C_{\tau_1}(.)$
and also the points
$\xi_j:=C_{\tau_1}(\zeta_j)$,
$\overline{\xi_j}=C_{\tau_1}(\zeta_j^*)$
and $t_j=C_{\tau_1}(e^{i\tau_j})$.

Then $\tau_2<\ldots<\tau_n$ from the interval $(\tau_1,\tau_1+2\pi)$
is a (real) critical point of $\widetilde{W}$
if and only if $t_2<\ldots<t_n$
is a (real) critical point of $V=V(t_2,\ldots,t_n)$.
\end{proposition}
\begin{proof}
Basically,
we use the chain rule to show that the critical points correspond to each other
under the diffeomorphism given by
the  inverse Cayley transform.

Let $\psi(\tau):=e^{i\tau}$.
It is  standard to see
\begin{equation*}
C_\theta(\psi(\tau))=
i\frac{1+e^{i(\tau-\theta)}}{1-e^{i(\tau-\theta)}}
= - \cot \frac{\tau-\theta}{2}, \quad
\frac{d}{d\tau}
C_\theta(\psi(\tau))
= \frac{1}{\sin^2 \frac{\tau-\theta}{2}}
\end{equation*}
where we used real differentiation with respect to
$\tau$.
We write $\Psi(\tau_2,\ldots,\tau_n):=\left(\psi(\tau_2),\ldots,\psi(\tau_n)\right)$
and
$\cpw(z_2,\ldots,z_n):=\left(C_\theta(z_2),\ldots,C_\theta(z_n)\right)^T$, where $\cdot^T$ denotes transpose.
Hence $\cpw\circ \Psi$ maps from $\bR^{n-1}$ to $\bR^{n-1}$
and $\widetilde{W}=W\circ \Psi=V\circ\cpw\circ \Psi+c$,
by Proposition \ref{energtrfalllim}.
The derivative of
$\cpw\circ \Psi$
as a real mapping is the diagonal matrix $D:=\mathrm{ diag}\left(\sin^{-2}
\left(\frac{\tau_2-\theta}{2}\right),\ldots,\sin^{-2}\left(\frac{\tau_n-\theta}{2}\right)\right)$.
This is an invertible matrix, because $\theta=\tau_1<\tau_2<\ldots<\tau_n<\tau_1+2\pi$.
Because of chain rule,
\begin{equation*}
\nabla_{\tau_2,\ldots,\tau_n} \widetilde{W}
=
\left.
\nabla_{t_2,\ldots,t_n} V
\right|_{\cpw\circ \Psi} \cdot D,
\end{equation*}
or by coordinates

\begin{equation*}
\left.
  \frac{\partial\widetilde{W}(\tau_2,\ldots,\tau_n)}{\partial \tau_j}=\frac{\partial V(t_2,\ldots,t_n)}{\partial t_j} \right|_{\cpw\circ \Psi} \cdot \frac{1}{\sin^2\left(\frac{\tau_j-\theta}{2}\right)}, \qquad j=2,\ldots,n,
\end{equation*}
which immediately implies the assertion.
\end{proof}

\section{Proofs of the two main theorems}

\begin{proof}[Proof of Theorem \ref{thm:mainthm1}]

We have
that $\tau_j$'s are different, and
$a_1,\ldots,a_n\in \bD$ is a sequence with $\zeta_j\ne 0$.
These imply that $(\exp(i \tau_1(\delta)),\ldots,\exp(i \tau_n(\delta)),\zeta_1,\ldots,\zeta_{n-1})$ is not in $E$ (see \eqref{diszkretkiveteleshalmaz}).
We also use the parametrization of 
 the solution curve $S$ defined in
\eqref{biggersolcurve},
and the strict monotonicity and continuity of $\delta\mapsto \tau_1(\delta)$.
Hence for any $w_1$,
$w_1=e^{i \beta}$
where $\beta\in[0,2\pi)$,
the respective points
on 
the solution curve $S$
are uniquely determined:
$w_j=w_j(w_1)$, more precisely,
$w_j=\exp(i \tau_j(\tau_1^{-1}(\beta)))$,
$j=2,\ldots,n$.

Fix $w_1$, or, equivalently, $\beta\in [0,2\pi)$. 
Now we want to show that
\[
\left(\tau_2,\tau_3,\ldots,\tau_n\right) \mapsto
\widetilde{W}(\beta,\tau_2,\tau_3,\ldots, \tau_n)
\]
(assuming $\beta< \tau_2< \ldots<\tau_n< \beta+2\pi$)
has only one critical point,
namely the point with $\tau_j=\tau_j(\tau_1^{-1}(\beta))$ for $j=2,3,\ldots,n$, which happens to be the unique minimum point in $\left(\tau_2,\tau_3,\ldots,\tau_n\right)$.

To this end, we are going to transform 
the question to the upper half-plane, 
as we want to use Lemma 6 from  \cite{SemmlerWegert}. 
We apply first the inverse Cayley transform $C(.)=C_{\beta}(.)$
which maps $w_1$ to $\infty$.
Hence we have $n-1$ pairs of fixed protons, $\xi_j=C(\zeta_j)$, $\overline{\xi_j}=C(\zeta_j^*)$,
$j=1,\ldots,n-1$
and $n-1$ free electrons on the real axis, $t_j=C(e^{i\tau_j})$, $j=2,\ldots,n$.
We know that $\beta< \tau_2< \ldots<\tau_n< \beta+2\pi$,
and $t_2< t_3<\ldots < t_n$ are equivalent.
(If any two of the $\tau$'s were equal, then the corresponding $t$'s would be equal too
and $\widetilde{W}(\tau_2,\ldots,\tau_n)=V(t_2,\ldots,t_n)=+\infty$, but
we assumed that $(w_1,\ldots,w_n,\zeta_1,\ldots,\zeta_{n-1}) \not\in E$
so that all $w_j$'s have to be different.)
Again, since we are outside $E$,
we know that $\xi_j\ne -i$
and $\overline{\xi_j}\ne -i$,
which, in turn, implies that $c$ is finite in \eqref{c_konst_diszkr}.
Thus, we can apply Proposition \ref{energtrfalllimdiscr}
(for $\ell=1$)
to relate the energy $\widetilde{W}$ on the unit circle and the energy $V$ on the real axis:
\begin{equation*}
\widetilde{W}(\beta,\tau_2,\ldots,\tau_n)
=
W(e^{i \beta},e^{i\tau_2},\ldots,e^{i\tau_n})
=
V(t_2,\ldots,t_n)+c.
\end{equation*}

Introducing
$U:=\{(t_2,\ldots,t_{n})\in \mathbb{R}^{n-1}:  t_2<t_3< \ldots <t_n\}$,
Lemma 6 from \cite{SemmlerWegert} gives
that there is exactly one critical point
$(\widetilde{t}_2,\ldots,\widetilde{t}_n)$
of $V$ in $U$ (gradient of $V$ vanishes), which is the global minimum point in $U$.
In view of Proposition \ref{pr:criticpoints}, the corresponding 
$(\widetilde{\tau_2},\ldots,\widetilde{\tau_n})$
with $\beta<\widetilde{\tau_2}<\ldots<\widetilde{\tau_n}<\beta+2\pi$
and $\exp( i\widetilde{\tau_2})=C_\beta^{-1}(\widetilde{t_2}),\ldots, \exp( i\widetilde{\tau_n})=C_\beta^{-1}(\widetilde{t_n})$,
is the only critical point of $\widetilde{W}=\widetilde{W}(\beta,\tau_2,\ldots,\tau_n)$, restricted to the simplex  $\Delta_\beta$
of points of the form $(\beta,\tau_2,\ldots,\tau_n)$
under the condition $\beta<\tau_2<\tau_3<\ldots<\tau_n <\beta+2\pi$. 
Note that $\Delta_\beta =Z_\beta \cap A$ 
with $Z_\beta$ denoting the hyperplane $\{\beta\} \times {\mathbb R}^{n-1}$.
Furthermore, applying 
Proposition \ref{energtrfalllimdiscr}, 
we get that this is the unique global minimum point 
of $\widetilde{W}$ on $\Delta_\beta$. 

Let us define
$\varphi:[0,2\pi)\rightarrow \mathbb{R}^n$ 
by putting $\varphi(\beta):=(\beta,\widetilde{\tau_2},\widetilde{\tau_3},\ldots,\widetilde{\tau_n})$.

As 
$S$
is a continuous curve lying in $A$, 
there exists a point 
$\mathbf{t}$ of 
$S \cap Z_\beta$,
which necessarily belongs to 
$S \cap Z_\beta \cap A = S \cap \Delta_\beta$,
too. 
However -- as it was shown  
in Theorem 4 in \cite{PapSchipp2015} -- 
$\nabla \widetilde{W} \equiv \mathbf{0}$ 
on $S$, 
therefore $\mathbf{t}$ is also a critical 
point of $\widetilde{W}|_{\Delta_\beta}$. 
Whence $\mathbf{t} = \varphi(\beta)$, 
the unique critical point of $\widetilde{W}|_{\Delta_\beta}$, 
which is, as said above, 
the global minimum point of $\widetilde{W}|_{\Delta_\beta}$, too.

It is easy to see that
$ \Phi:=W\circ\varphi$ 
is continuous on 
$[0,2\pi)$
and with $ \Phi(2\pi):=W(\varphi(0))$ is continuously extensible onto $[0,2\pi]$. 
Thus $\Phi=W\circ \varphi$ 
has a global minimum on
$[0,2\pi)$,
let it be $\beta^*$.
Obviously, $\varphi({\beta^*})$ is also on the solution curve $S$,
and $\widetilde{W}(\tau_1,\ldots,\tau_n)$ has
a global minimum in $\varphi({\beta^*})$.
Since $S$ is a smooth
arc, and $\nabla \widetilde{W} \equiv {\bf 0}$ on $S$, 
we get that $\left. \widetilde{W}\right|_{S}\equiv const$.
That is, we find $\left. \widetilde{W}\right|_{S}\equiv \varphi({\beta^*})$, the global minimum of
the discrete energy function $\widetilde{W}=\widetilde{W}(\tau_1,\ldots,\tau_n)$.

Finally, we show that all points of $S_\bR$ are global minimum points of $\widetilde{W}(.)$.
Using that $\widetilde{W}(.)$ is $(2\pi,\ldots,2\pi)$-periodic,
that is $\widetilde{W}(\tau_1,\tau_2,\ldots,\tau_n)=\widetilde{W}(\tau_1+2\pi,\tau_2+2\pi,\ldots,\tau_n+2\pi)$
and that for  each $j$, $\tau_j(\delta+2n\pi)=\tau_j(\delta)+2\pi$,
we obtain that $\widetilde{W}(\tau_1(\delta),\ldots,\tau_n(\delta))$ is actually $2n\pi$ periodic in $\delta$.
This, expressed with $S$ and $S_\bR$,
implies that all points of $S_\bR$ are global minimum points of $\widetilde{W}(.)$.
\end{proof}

Note that the above provides a positive answer to the
question raised in \cite{PapSchipp2015}, p.~476:
the discrete energy function $\widetilde{W}=\widetilde{W}(\tau_1,\ldots,\tau_n)$
attains global minimum
at every point of the full solution curve $S_\bR$.
Moreover, these are the only critical points of $\widetilde{W}$.

\bigskip

We collect the following set of "bad" configurations:
\begin{multline}
X:=\{(z_1,z_2,\ldots,z_n)\in (\partial \mathbb{D})^n:
z_j=z_k \text{ for some }j\ne k,
 \text{ or }
B'(0)=0
\}.
\end{multline}

\begin{proof}[Proof of Theorem \ref{thm:reverseproblem}]
Let $(z_1, \ldots, z_n)\in (\partial \mathbb{D})^n \setminus X$ be given.
Denote their arguments by $t_j:=\Im \log(z_j)$, $j=1,2,\ldots,n$.
Without loss of generality, we may
assume that $t_1,t_2,\ldots,t_n\in[0,2\pi)$ and $t_1<t_2<\ldots<t_n$.

We use the above cited result of Hjelle
providing 
a Blaschke product $B(z)=B(z_1,\ldots,z_n;z)$ with degree $n$,
satisfying \eqref{eq:blaschkeinterpol}.
Denote the leading coefficient
of $B(.)$ by $\chi$
where $\chi=e^{i\delta_0}$; note that $\delta_0$ is determined only $\bmod~2\pi$ by this choice.
Let us define $B_1(z):=\chi^{-1} B(z)$ 
which is the monic Blaschke
product with the same zeros.
We use $\alpha$, $T$, $S_0$, $S$ and $S_\bR$
defined for $B_1(.)$.
Now we fix the value of $\delta_0$ so 
that $-\delta_0\in[\alpha,\alpha+2\pi)$; 
observe that this does not change the value of $\chi$
and does not cause circular dependence.
Note that the sets $S_\bR$ defined for
$B$ and $B_1$ are the same,
because multiplying the Blaschke product with a constant 
is just a translation of variable.
More precisely
$\tau_j(B;\delta)=\tau_j(B_1;\delta-\delta_0)$ for all $j=1,2,\ldots,n$, $\delta\in\bR$.

Hjelle's result means that
$\tau_j(B;0)=t_j$, 
hence $\tau_j(B_1;-\delta_0)=t_j$.
By the choice of $\delta_0$,
we immediately see that
$(t_1,t_2,\ldots,t_n)=T(-\delta_0)$,
that is,
$(t_1,t_2,\ldots,t_n)$
is on $S_0$ defined 
in \eqref{solcurve} for the monic Blaschke product $B_1(.)$.

We use the description from Theorem \ref{thm:mainthm1}.
This way we obtain that
$\widetilde{W}(.)$ has global minimum
at the points $T(\delta)$,
$\delta\in [\alpha,\alpha+2\pi)$
(defined by 
$B_1(.)$).
Observe that when the parameter $\delta$ 
changes continuously further on in 
$[\alpha,\alpha+2n\pi)$, 
the curve $T(\delta)$ recovers ($\bmod ~2\pi$) 
the same set of arguments
$(t_1,\ldots,t_n)$ $n$ times,
in each cyclic permutations of them, 
while the corresponding $z_1,\ldots,z_n$ is 
repeated $n$ times 
(in each cyclic order of the values) 
always determining the same Blaschke product.

We remark, that according to Proposition \ref{szimmenergiakorvonal},
the energy function $W(.)$ has critical point in $(z_1,z_2,\ldots,z_n)$ not just with restriction to the unit circle,
but  also in the total electrostatic equilibrium sense.
This was also observed in \cite{PapSchipp2001}, see Theorem 2.
\end{proof}

Roughly speaking,
the union of solution curves
for different $a_1,a_2,\ldots,a_n$ covers
the whole 
$A\cap Q$,
and considering as electrons on the unit circle,
the whole
space  $(z_1,z_2,\ldots,z_n)\in (\partial \mathbb{D})^n\setminus X$.

This last result, when compared with Theorem \ref{thm:mainthm1},
shows a direct relation between the location of electrons, $z_1,z_2, \ldots, z_n$ and
the location of pairs of protons,
$\zeta_1,\zeta_1^*,\zeta_2,\zeta_2^*,\ldots,\zeta_{n-1},\zeta_{n-1}^*$.

\begin{corollary}
If $(z_1, \ldots, z_n)\in (\partial \mathbb{D})^n\setminus X$
is given, then the points $\zeta_1, \ldots, \zeta_{n-1}\in \mathbb{D}\setminus \{0\}$ in Theorem \ref{thm:reverseproblem}
are the critical points of the Blaschke product satisfying \eqref{eq:blaschkeinterpol}.
\end{corollary}

\section*{Acknowledgments}

The authors gratefully acknowledge their indebtedness to Margit Pap and Ferenc Schipp for
calling their attention to the problem and for useful suggestions and discussions.

This research was partially supported by the DAAD-TKA Research Project "Harmonic Analysis and Extremal Problems" \# 308015.

Marcell Gaál was supported by the National, Research and Innovation Office NKFIH Reg. No.'s K-115383 and K-128972,
and also by the Ministry for Innovation and Technology, Hungary throughout Grant TUDFO/47138-1/2019-ITM.

Béla Nagy was supported  by the ÚNKP-18-4 New National Excellence Program of
the Ministry of Human Capacities.

Zsuzsanna Nagy-Csiha was supported by the ÚNKP-19-3 New National Excellence Program of the Ministry for Innovation and Technology.

Szilárd Gy.~Révész was supported in part by Hungarian National Research, Development and Innovation Fund projects \# K-119528 and K-132097.

The authors are grateful to the anonymous referees for their thorough work, precise corrections and useful suggestions.

The authors are thankful to Gunter Semmler for his interest and constructive remarks.

\bibliographystyle{amsplain}
\bibliography{mainarxivv2}

Marcell Gaál\\

\smallskip

Béla Nagy\\
 MTA-SZTE Analysis and
 Stochastics Research Group,\\
  Bolyai Institute, University of Szeged\\
  Aradi vértanuk tere 1\\
   6720 Szeged, Hungary\\
\href{mailto:nbela@math.u-szeged.hu}{nbela@math.u-szeged.hu}

\smallskip

Zsuzsanna Nagy-Csiha\\
Institute of Mathematics and Informatics,
 Faculty of Sciences,\\
University of Pécs\\
Ifjúság útja 6\\
7624 Pécs,  Hungary\\
\qquad and\\
Department of Numerical Analysis, Faculty of Informatics, \\
Eötvös Loránd University \\
Pázmány Péter sétány 1/c \\
1117 Budapest, Hungary\\
\href{mailto:ncszsu@gamma.ttk.pte.hu}{ncszsu@gamma.ttk.pte.hu}

\smallskip

Szilárd Gy.{} Révész\\
 Alfréd Rényi Institute of Mathematics\\
 Reáltanoda utca 13-15\\
 1053 Budapest, Hungary \\
 \href{mailto:revesz.szilard@renyi.hu}{revesz.szilard@renyi.hu}

\end{document}